\documentclass[11pt,reqno]{amsart}

\usepackage{amsmath,amsthm,amssymb}
\usepackage{enumerate}
\usepackage{graphicx}
\newtheorem{theorem}{Theorem}[section]

\newtheorem{prop}{Proposition}[section]

\theoremstyle{remark}

\numberwithin{equation}{section}
 \DeclareMathOperator\hdim{\dim_H}
\DeclareMathOperator\adim{\dim_A}
\DeclareMathOperator\fdim{\dim_F}
\DeclareMathOperator\spt {spt}
\def\N{\mathbb{N}}

\begin{document}

\title[On Assouad dimension and arithmetic progression]{On Assouad dimension and arithmetic progressions in sets defined by digit restrictions}

\author[Jinjun Li]{Jinjun Li}

\address[Jinjun Li]{School of Mathematics and Information Science,
 Guangzhou University, Guang\-zhou, 510006, P.~R.~China}
\email{li-jinjun@163.com}


\author{Min Wu*}
\address[Min Wu]{Department of Mathematics, South China
 University of Technology, Guang\-zhou, 510641, P.~R.~China}
\email{wumin@scut.edu.cn}

\author{Ying Xiong}
\address[Ying Xiong]{Department of Mathematics, South China
 University of Technology, Guang\-zhou, 510641, P.~R.~China}
\email{xiongyng@gmail.com}

\thanks{*Corresponding author. Email: wumin@scut.edu.cn}




\begin{abstract}
We show that the set defined by digit restrictions contains arbitrarily long arithmetic progressions if and only if its Assouad dimension is one. Moreover, we show that for any $0\le s\le 1$, there exists some set on $\mathbb{R}$ with Hausdorff dimension $s$ whose Fourier dimension is zero and it contains arbitrarily long arithmetic progressions.
\end{abstract}

\maketitle
\section{Introduction}
For $k\ge 3$, we call $A\subset \mathbb{R}$ a $k$-term arithmetic progression with gap $\delta$ if there exist some $t>0$ such that
\[
A=\{t+\delta x: x=0,\cdots, k-1\}.
\]
We call a set contains arbitrarily long arithmetic progressions if it contains a $k$-term arithmetic progression for any $k\ge 3.$
In the fields of combinatorial theory and number theory, we usually want to know whether a set can contain arithmetic progressions and therefore it is important to find conditions guaranteeing the existence of arithmetic progressions. In the discrete case, a theorem of Roth \cite{Ro} states that if a subset $A\subset \N$ has positive upper density, i.e.,
\[
\overline{d}(A)=\limsup_{n\to \infty}\frac{\sharp(A\cap\{1,2\cdots, n\})}{n}>0,
\]
then $A$ must contain a non-trivial $3$-term arithmetic progression. Here and in the sequel, $\sharp(E)$ denotes the number of elements in $E.$  Roth's result partially solved a conjecture related to sets containing arithmetic progressions posed by Erd\H{o}s and Tur\'{a}n \cite{ET}.  Later, Szemer\'{e}di \cite{Sz1,Sz2} extended this to arbitrarily long arithmetic progressions. It is worth to mention that Furstenberg \cite{Fur} gave a second proof of Szemer\'{e}di's theorem by employing ergodic theory. However, Szemer\'{e}di's theorem can also holds for sets with zero density, for example, Green and Tao \cite{GT} showed that the primes contain arbitrarily long arithmetic progressions. In the continuous case,  {\L}aba and Pramanik \cite{LP} recently initiated the study on the relationship between Fourier decay and the existence of arithmetic progressions. They proved that, if $E\subset \mathbb{R}$ is a closed set with Hausdorff dimension close to one and $E$ supports a probability measure obeying appropriate Fourier decay and mass decay, then $E$ contains non-trivial 3-term arithmetic progressions. Their results attract many scholars to study the connection between Fourier decay and the existence of arithmetic progressions, see \cite{Car,CLP, FSY, CKL,Po,Po2, Shm} and the references therein. In particular, Shmerkin \cite{Shm} showed that there exist Salem sets without any $3$-term arithmetic progressions, and Lai \cite{CKL} showed that there exist a prefect set with zero Fourier dimension but containing arbitrarily long arithmetic progressions.

 There is a natural question: under what conditions can a set in $\mathbb{R}$ contain arbitrarily long arithmetic progressions? Fraser and Yu \cite{FY} showed that a bounded set cannot contain arbitrarily long arithmetic progressions if it has Assouad dimension strictly smaller than one. Fraser and Yu's result extended a result of Dyatlov and Zahl \cite{DZ}, which states that any Ahlfors-David regular set of dimension less than one cannot contain arbitrarily long arithmetic progressions. Recall that we call a set $E$ an Ahlfors-David regular set of dimension $\alpha$ if there exists a constant $C\ge 1$ and a Borel probability measure $\mu$ on $E$ such that
 \[
 C^{-1} r^\alpha\le \mu(B(x,r)\cap E)\le Cr^\alpha
 \]
for any $x\in E$ and $r>0.$ It is well known that Ahlfors-David regular set is a kind of nice set in the sense that its Hausdorff and Assouad dimensions are equal. However, the Fourier dimension may be different from the Hausdorff dimension for an Ahlfors regular set. For example, the classical Cantor ternary set has Fourier dimension $0$ and Hausdorff dimension $\log 2/\log 3$, respectively.

 Fraser and Yu \cite{FY} also showed that the converse of their result is not true by constructing examples. For example, let $E=\{1, 1/2^3, \cdots, 1/n^3,\cdots\}\subset \mathbb{R}.$ Then we can check that  $\adim E=1$, but it contains no arithmetic progressions of length $3.$ However, they proved that a set in $\mathbb{R}$ asymptotically contains arbitrarily long arithmetic progressions (in the sense that it does not give strict containment of arithmetic
progressions) if and only if it has full Assouad dimension.

  In this paper, we will consider the set defined by digit restrictions and show that it contains arbitrarily long arithmetic progressions if and only if its Assouad dimension is one. Moreover, we show that for any $0\le s\le 1$, there exists some set in $\mathbb{R}$ with Hausdorff dimension $s$ whose Fourier dimension is zero and it contains arbitrarily long arithmetic progressions.

 \section{Main results}
Let us first recall the definitions of set defined by digit restrictions and Assouad dimension.  Let $b\ge 2$ be an integer and $D\subset \{0, 1,\cdots b-1\}$ a nonempty proper subset. Let $S\subset \N$ be an infinite set, and let $E_{S,D}\subset [0,1]$ be the compact set
\[
E_{S,D}=\left\{\sum_{n=1}^\infty x_nb^{-n}: \text{$x_n\in D$ if $n\notin S$ and  $x_n\in\{0,1,\cdots,b-1\}$ otherwise}
\right\}.
\]
The set $E_{S,D}$ is called a set defined by digit restrictions. It is well known that the Hausdorff dimension of $E_{S,D}$ is closed related to
 the lower density of $S$. More precisely,
\begin{equation}\label{hd}
\hdim E_{S,D}=\underline{d}(S,D),
\end{equation}
where
\[
\underline{d}(S,D)=\frac{\log (\sharp D)}{\log b}+\left(1-\frac{\log (\sharp D)}{\log b}\right)
\liminf_{N\to \infty}\frac{\sharp(S\cap\{1,\cdots, N\})}{N}.
\]
Here and in the sequel, $\hdim E$ denotes the Hausdorff dimension of the set $E$.  For more details about Hausdorff dimension and the theory of fractal dimensions, we refer the reader to the famous book \cite{Fal}. And for the proof of $\eqref{hd}$ and other dimensional properties and applications of $E_{S,D}$ in the case $b=2$ and $D=\{0\}$, see \cite{BPe,BP,LW2}.

Assouad dimension, introduced by Assouad \cite{Ass}, provides quantitative information on the local behaviour of the geometry of the underlying set.  More precisely, Let $X$ be a metric space and $E\subset X$. For $r, R>0$, let $N_r(E)$ denote the least number of balls with radii equal to $r$ needed to cover the set $E$ and let
\[
N_{r, R}(E)=\sup_{x\in E}N_r(B(x,R)\cap E).
\]
Then the Assouad dimension of $E$ is defined as
\[
\begin{split}
\adim E=\inf\bigg\{\alpha\geq 0 : \quad &\text{there are constants $b,C>0$ such that for any} \\
&\text{ $0<r<R<b$, $N_{r,R}(E)\leq C\left(\frac{R}{r}\right)^{\alpha}$ holds}\bigg\}.
\end{split}
\]

It is well known that $\hdim E\le \adim E$ for any bounded set $E\subset \mathbb{R}$. For other basic properties of Assouad dimension, see \cite{Luu, Rob}. In particular, The Assouad dimension of $E_{S,D}$ is
\[
\begin{split}
\adim E_{S,D}=&\frac{\log (\sharp D)}{\log b}\\
              &+\left(1-\frac{\log (\sharp D)}{\log b}\right)\limsup\limits_{m\to \infty}\sup_{k\ge 1}\frac{\sharp(S\cap\{k+1,\cdots, k+m\})}{m},
\end{split}
\]
see \cite{DWW,LLMX}.

 It is worth to point out that the Assouad dimension plays an important role in the theory of embeddings of metric spaces in Euclidean spaces and in the study of quasimmetric mappings, see \cite{Hei,Luu,Rob}.

The following result establishes the relationship between the Assouad dimension and the existence of arithmetic progressions for the set defined by digit restrictions.
\begin{theorem}\label{MRR1}
Let $E_{S,D}$ be the set defined by digit restrictions. Then, $E_{S,D}$ contains arbitrarily long arithmetic progressions if and only if $\adim E_{S,D}=1$.
\end{theorem}

 The Fourier dimension of a set is  a measure of exactly how rapid the decay of the Fourier transformation of the measures supported on it. More precisely,
the Fourier dimension of $A\subset \mathbb{R}$ is defined as
\[
\begin{split}
\fdim A=\sup\Big\{&\beta\in[0,1]:  \quad \text{there exist a probability measure $\mu$ on $E$ and }  \\
&\text{ constant $C>0$, such that $|\widehat{\mu}(\xi)|\le C|\xi|^{-\beta/2}$ for all $\xi \in \mathbb{R}$}\Big\},
\end{split}
\]
where $\widehat{\mu}(\xi)=\int e^{-2\pi i\xi x}d\mu(x)$ is the Fourier transformation of $\mu$. It is well known that $\fdim A\le \hdim A$ for any Borel set $A\subset \mathbb{R}$ and $A$ is called a Salem set if $\fdim A= \hdim A,$ see \cite{Ma}.

Roughly speaking, {\L}aba and Pramanik's  result can be interpreted as saying that large Fourier
decay together with large power mass decay imply the existence of progressions. However, we will show that, for nontrivial set defined by digit restrictions, its Fourier dimension is zero. To avoid the trivial case, we require that $\N\setminus S$ is not finite. In fact, if $\N\setminus S$ is a finite set, by the construction of $E_{S,D}$, it will be trivial in the sense that it contains a nontrivial interval and therefore its Fourier dimension is one. However, the Fourier dimensions of sets defined by digit restriction are all zero except the trivial cases. This is the content of the following result.
\begin{theorem}\label{MRR2}
Suppose that $\N\setminus S$ is not finite. Then, $\fdim E_{S,D}=0$.
\end{theorem}

Combining the above two results and the formula of the Hausdorff dimension of $E_{S,D}$, we can obtain the following result.
\begin{theorem}\label{MRR3}
For any $0\le s\le 1$, there exists a compact set $E\subset \mathbb{R}$ such that $\hdim E=s, \fdim E=0$,  and it contains arbitrarily long arithmetic progressions.
\end{theorem}
It is worth to point out that Lai \cite{CKL} obtained the same result as above Theorem \ref{MRR3} with the help of some special Moran sets.
\section{Proofs }
This section is devoted to the proofs of our results.

\begin{proof}[Proof of Theorem $\ref{MRR1}$.]
Due to the Lebesgue density theorem, any subset on $\mathbb{R}$ with positive Lebesgue measure must contain arbitrary long arithmetic progressions.  On the other hand, it is well known that the Assouad dimension of nontrivial interval is one. Therefore, if $\N\setminus S$ is finite, the claim in Theorem \ref{MRR1} holds since $E_{S,D}$ contains a nontrivial interval.

Next, we assume that $\N\setminus S$ is not finite.

If $E_{S,D}$ contains arbitrarily long arithmetic progressions, then, for any $n\ge 3$ there exist $x_n$ and $t_n>0$ such that $x_n, x_n+t_n, \cdots, x_n+nt_n\in E_{S,D}$. Since $E_{S,D}$ is compact, we have $\sup_n nt_n<\infty$ and therefore $t_n\to 0$ as $n\to \infty.$ Put $R_n=nt_n$ and $r_n=t_n$. It is easy to check that
\[
N_{r_n, R_n}( E_{S,D})\ge n=\left(\frac{R_n}{r_n}\right)^1,
\]
which implies that $\adim E_{S,D}\ge 1$ and therefore $\adim E_{S,D}= 1.$ Let us remark that the claim of this part can be proved by the result of Fraser and Yu \cite{FY} directly.

On the other hand, suppose that $\adim E_{S,D}=1,$ that is,
\[
\frac{\log (\sharp D)}{\log b}+\left(1-\frac{\log (\sharp D)}{\log b}\right)\limsup\limits_{m\to \infty}\sup_{k\ge 1}\frac{\sharp(S\cap\{k+1,\cdots, k+m\})}{m}=1.
\]
It implies that
\begin{equation}\label{ad}
\limsup\limits_{m\to \infty}\sup_{k\ge 1}\frac{\sharp(S\cap\{k+1,\cdots, k+m\})}{m}=1.
\end{equation}

\begin{figure}
\centering
{\includegraphics[width=.8\textwidth]{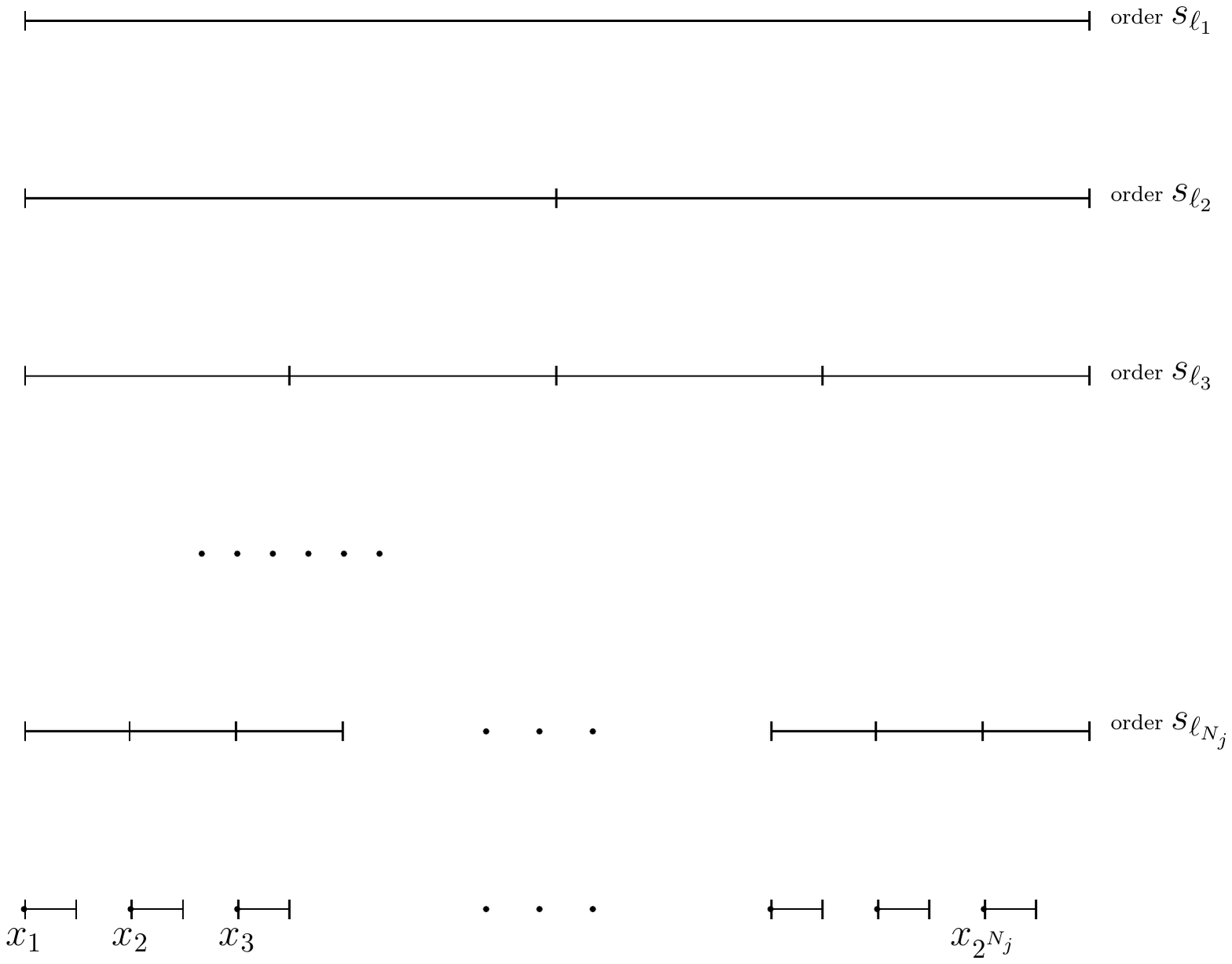}}\\[-28ex]
\parbox{.8\textwidth}{\footnotesize \textbf{Fig.1}
\centering
 The construction of the arithmetic progression.}
\label{fig:1}
\end{figure}
 \bigskip

Recall that $S=\{s_1, s_2,\cdots, s_n,\cdots\}$. For any positive integers $m, n$ with $m<n$, define
\[
\begin{split}
N_m^n(S)=\max\big\{\ell: &\quad \text{there exists $i\ge 1$ such that $s_{i+j}=s_{i+j-1}+1$}\\
&\text{ for any $1\le j\le \ell$ and $s_k\in[m, n]$ for any $i\le k\le i+\ell$}\big\}.
\end{split}
\]
That is to say, $N_m^n(S)$ is the maximum number of successive integers among the elements of $S$ falling in the interval $[m,n]$.

We claim that, by \eqref{ad},  there exist $m_j\nearrow +\infty$ and $\{k_j\}\subset \N$ such that
\begin{equation}\label{cond1}
\text {$N_{k_{j}+1}^{k_j +m_j}(S)\to +\infty$ as $j\to \infty.$}
\end{equation}
In fact, if there exists some constant $M>0$ such that
\[
\text {$N_{k+1}^{k +m}(S)\le M$ for any $m\ge 1$ and $k\ge 1$,}
\]
then, for any $k\ge 1$ and large enough $m$
\[
\frac{\sharp(S\cap\{k+1,\cdots, k+m\})}{m}\le \frac{M\cdot [\frac{m}{M+1}]}{m}\le \frac{M}{M+1}<1,
\]
which is contradictory to the assumption \eqref{ad}.

For any $n\ge 2,$  by $\eqref{cond1}$, there exists some $j=j(n)$ such that
\begin{equation}\label{aa}
N_{k_{j}+1}^{k_j +m_j}(S)\ge \frac{\log n-\log(\sharp D)}{\log b}.
\end{equation}
Write  $N_j=N_{k_{j}+1}^{k_j +m_j}(S).$ By the definition, there exist some $s_{\ell_1}, s_{\ell_2}\cdots, s_{\ell_{N_j}}\in [k_j+1, k_j+m_j]$ such that $s_{\ell_{j+1}}=s_{\ell_j}+1$ for any $1\le j\le N_j-1$ and $s_{\ell_{N_j}}+1\notin S.$

By the definition, the set $E_{S,D}$ hits exactly $M_n$ $b$-adic intervals of generation $n$, where
\[
M_n=b^{\sharp(S\cap\{1,2,\cdots,n\})}( \sharp D)^{\sharp((\N\setminus S)\cap \{1,2,\cdots,n\})}.
\]
 These $b$-adic intervals are called the basic intervals of order $n$. Take any $d\in D$. Let
 \[
 I_{s_{\ell_1}}=\left(\frac{d}{b^{s_{\ell_1}}}, \frac{d+1}{b^{s_{\ell_1}}}\right).
 \]
 By the construction of $E_{S,D}$ and the definition of $N_j$, the interval $I_{s_{\ell_1}}$ contains $(\sharp D)b^{N_j}$ basic intervals of order $s_{\ell_{N_j}+1}$. Let $x_1, x_2, \cdots, x_{(\sharp D)b^{N_j}}$ be the left endpoints of these intervals from left to the right. Fig.1 illustrates the construction of the points $x_1, \cdots, x_{(\sharp D)b^{N_j}}$ in the special case that $b=2$ and $D=\{0\}$. Then, it is easy to check that all these points are in $E_{S,D}$, and for any $1\le i\le (\sharp D)b^{N_j}-1$,
\[
|x_{i+1}-x_i|=b^{-s_{\ell_{N_j}}-1}.
\]
 It follows from $\eqref{aa}$ that $(\sharp D)b^{N_j}\ge n$ and therefore $E_{S,D}$ contains a arithmetic progressions of length $n$. By the arbitrariness of $n$ we claim that $E_{S,D}$ contains arbitrarily long arithmetic progressions.
\end{proof}

Next we shall prove Theorem $\ref{MRR2}$. It follows from the definition of the Fourier dimension and the following proposition immediately.
\begin{prop}\label{pro}
Suppose that $\N\setminus S$ is not finite and  $\nu$ is a probability measure supported on $E_{S,D}.$ Then,
\[
\limsup_{|n|\to \infty}|\widehat{\nu}(n)|>0,
\]
where
\[
\widehat{\nu}(n)=\int e^{-2\pi in t}d\nu(t), n\in \mathbb{Z}.
\]
\end{prop}
\begin{proof}
First we consider a special class of sets defined by digit restrictions. Define
\[
\mathcal{E}=\left\{E_{S,D}:\text{$\N\setminus S$ is not finite and $1\notin S$}\right\}.
\]
We claim that for any $E\in \mathcal{E}$ and any probability measure $\mu$ supported on $E$,
\[
\limsup_{|n|\to \infty}|\widehat{\mu}(n)|>0,
\]
The idea of the proof is essentially due to Salem and Zygmund \cite{SZ} (see also \cite[p.~113]{Ma}).

Suppose for a contradiction that
\begin{equation}\label{sump}
\lim_{|n|\to \infty}\widehat{\mu}(n)=0.
\end{equation}
Take any $d\in\{0,1,\cdots,b-1\}\setminus D$.   Define
\[
I=\left(\frac{d}{b}, \frac{d+1}{b}\right).
\]
Since $\N\setminus S$ is not finite, there exists an increasing sequence of integers $\{k_j\}_{j=1}^\infty$ such that $k_1\ge k_0$ and $k_j+1\notin S$ for $j=1,2,\cdots.$ It is not difficult to check that, for any $x\in E_{S,D}$ and $j=1,2,
\cdots$,
\begin{equation}\label{bz}
\{b^{k_j}x\}\notin I.
\end{equation}
Here $\{y\}$ stands for the fractional part of $y$, that is, $\{y\} \in [0, 1)$.

In fact, for $x=\sum_{n=1}^\infty x_nb^{-n}\in E_{S,D}$ and any $j\ge 1$
\[
b^{k_j}x=b^{k_j}\sum_{n=1}^{k_j}\frac{x_n}{b^n}+b^{k_j}\sum_{n=k_j+1}^\infty\frac{x_n}{b^n}.
\]
Noting that $x_{k_j+1}\not= d$, we have $\{b^{k_j}x\}\notin I.$

Now choose a function $\varphi$ from the Schwartz class such that the support of $\varphi$ is contained in $I$, i.e.,  $\spt \varphi\subset I.$ That is, $\varphi$ is infinitely differentiable and its derivatives of all orders tend to zero at infinity more quickly than $|x|^{-k}$ for all integers $k.$ We can also require that $\int \varphi(x) dx=1.$ For any $j=1,2,
\cdots,$ define
\[
\varphi_j(x)=\varphi(\{b^{k_j}x\})\quad x\in [0,1].
\]
By \eqref{bz}, we have $\spt\varphi_j\cap E_{S,D}=\emptyset.$ It follows from the Fourier inversion formula that
\[
\varphi_j(x)=\varphi(\{b^{k_j}x\})=\sum_{\ell\in \mathbb{Z}}\widehat{\varphi}(\ell)e^{2\pi i \ell \{b^{k_j}x\}}=\sum_{\ell\in \mathbb{Z}}\widehat{\varphi}(\ell)e^{2\pi i \ell b^{k_j}x}.
\]
Therefore, we have $\widehat{\varphi}_j(b^{k_j}\ell)=\widehat{\varphi}(\ell)$ and the other Fourier coefficients of $\varphi_j$ are zeros. It follows from the Parseval formula that, for any $j\ge 1$ and any $m>1$,
\[
\begin{split}
0&=\int \varphi_j(x)d\mu(x)=\sum_{\ell \in \mathbb{Z}}\overline{\widehat{\varphi}_j(\ell)}\widehat{\mu}(\ell)\\
 &=\sum_{\ell \in \mathbb{Z}}\overline{\widehat{\varphi}_j(b^{k_j}\ell)}\widehat{\mu}(b^{k_j}\ell)
=\sum_{\ell \in \mathbb{Z}}\overline{\widehat{\varphi}(\ell)}\widehat{\mu}(b^{k_j}\ell)\\
 &=\overline{\widehat{\varphi}(0)}\widehat{\mu}(0)+\sum_{1<|\ell|<m}\overline{\widehat{\varphi}(\ell)}\widehat{\mu}(b^{k_j}\ell)+\sum_{|\ell|\ge m}\overline{\widehat{\varphi}(\ell)}\widehat{\mu}(b^{k_j}\ell).
\end{split}
\]
Next, we shall estimate the above three terms. The first term is just $\mu(E_{S,D})=1.$  Note that $\varphi$ is in the Schwartz class and
\[
\left|\sum_{|\ell|\ge m}\overline{\widehat{\varphi}(\ell)}\widehat{\mu}(b^{k_j}\ell)\right|\le \mu(E_{S,D})\sum_{|\ell|\ge m}|\widehat{\varphi}(\ell)|,
\]
the third term can be arbitrarily small by choosing $m$ large enough. Finally, for any $m$, by the assumption \eqref{sump} we have
\[
\left|\sum_{1<|\ell|<m}\overline{\widehat{\varphi}(\ell)}\widehat{\mu}(b^{k_j}\ell)\right|\le 2m \sup_{|n|\ge b^{k_j}, n\in \mathbb{Z}}|\widehat{\mu}(n)|\to 0
\]
as $j\to \infty.$

Combining the above estimates, we have $\mu(E_{S,D})=0,$ which is a contradiction and therefore the claim holds.

Finally, we assume that $\N\setminus S$ is not finite and $\nu$ is a probability measure supported on $E_{S,D}$.

Let $k_0=\min\{k: k\notin S\}$. The number $k_0$ does exist due to the fact that $\N\setminus S$ is not finite. Clearly, we can decompose $E_{S,D}$ as
\[
E_{S,D}=\bigcup_{i=0}^{b-1}\left(E_{S,D}\cap \left(\frac{i}{b^{k_0-1}},\frac{i+1}{b^{k_0-1}}\right)\right):=\bigcup_{i=0}^{b-1}J_i.
\]
Observing that $J_i\in\mathcal{E}$ for $i=0,1,\cdots, b-1$, we have
\[
\nu=\sum_{i=0}^{b-1}\nu_i,
\]
where $\nu_i=\nu|_{J_i}$, the restriction of $\nu$ on $J_i, i=0,1,\cdots,b-1.$ Therefore, it follows from the above argument and the fact $\widehat{\nu}=\sum_{i=0}^{b-1}\widehat{\nu_i}$ that
\[
\limsup_{|n|\to \infty}|\widehat{\nu}(n)|>0.\qedhere
\]

%
\end{proof}

\begin{proof}[Proof of Theorem $\ref{MRR3}$.]

First, in the case $b=2$ and $D=\{0\}$, it follows from \eqref{hd} that
\[
\hdim E_{S,\{0\}}=\liminf_{N\to \infty}\frac{\sharp(S\cap\{1,\cdots, N\})}{N}
\]
and
\[
\adim E_{S,\{0\}}=\limsup\limits_{m\to \infty}\sup_{k\ge 1}\frac{\sharp(S\cap\{k+1,\cdots, k+m\})}{m}.
\]
Therefore, by Theorems \ref{MRR1} and \ref{MRR2}, it is sufficient to construct a set $S\subset \N$ such that $\N\setminus S$ is not finite and satisfies
\begin{equation}\label{con1}
\liminf_{N\to \infty}\frac{\sharp(S\cap\{1,\cdots, N\})}{N}=s
\end{equation}
and

\begin{equation}\label{con2}
\limsup\limits_{m\to \infty}\sup_{k\ge 1}\frac{\sharp(S\cap\{k+1,\cdots, k+m\})}{m}=1.
\end{equation}
Then, define
\[
E=E_{S,\{0\}}.
\]

We divide the proof into the following three cases according to the values
of $s.$

\textsc{Case 1}: $0<s<1$. Let $\{M_n\}_{n\ge 1}$ be an increasing sequence of positive integers with $M_1=1$ and
\[
 s M_2>1, \quad \lim\limits_{n\to \infty}\frac{M_n}{M_{n+1}}=0.
\]
Next we will use the sequence $\{M_n\}$ to construct the desired subset $S.$ Let us remark that our construction are inspired by the method in \cite{DWW}.

For $i\ge 1,$ let
\[
S_{2i}=\left\{M_{2i}+1,  \cdots, M_{2i}+\left[\frac{2i-1}{2i}(M_{2i+1}-M_{2i})\right]-1, M_{2i+1}\right\}
\]
and
\[
S_{2i-1}=\left\{M_{2i-1}+1, \cdots, M_{2i-1}+[s M_{2i}]-L_{2i-1}-1, M_{2i}\right\},
\]
where
\[
L_{2i-1}=\sum_{j=1}^{i-1}[sM_{2j}]+\sum_{j=1}^{i-1}\left[\frac{2j-1}{2j}(M_{2j+1}-M_{2j})\right].
\]
Here and in the sequel, the notation $[x]$ denotes the integer part of $x$.

Then, define
\[
S=\bigcup_{i\ge1}S_i.
\]
Note that, for any $i\ge 1$,
\[
\left[\frac{2i-1}{2i}(M_{2i+1}-M_{2i})\right]\not= M_{2i+1}-M_{2i}.
\]
Therefore, $\N\setminus S$ is not finite.

On the other hand, it is easy to check that
\[
\text{$\sharp(S\cap\{1,\cdots, M_{2i}\})=[s M_{2i}]$ for any $i\ge 1$}
\]
and
\[
\limsup\limits_{m\to \infty}\sup_{k\ge 1}\frac{\sharp(S\cap\{k+1,\cdots, k+m\})}{m}\ge \sup_{i\ge 1}\frac{\sharp(S\cap S_{2i})}{M_{2i+1}-M_{2i}}=1.
\]
Therefore, $S$ satisfies conditions \eqref{con1} and \eqref{con2}.

\textsc{Case 2}: $s=0.$ Let
\[
S=\bigcup_{n=1}^\infty\{ n^3+1, \cdots, n^3+n\}.
\]
Then, it is easy to check that $S$ satisfies conditions \eqref{con1}, \eqref{con2} and $\N\setminus S$ is not finite.


\textsc{Case 3}: $s=1.$ The proof is similar to that in Case 1.
We only give the key constructions.

Let $\{M_n\}_{n\ge 1}$ be an increasing sequence of positive integers with $M_1=1$ and
\[
 \lim\limits_{n\to \infty}\frac{M_n}{M_{n+1}}=0.
\]
For $i\ge 1,$ let
\[
S_{2i}=\left\{M_{2i}+1,  \cdots, M_{2i}+\left[\frac{2i}{2i+1}(M_{2i+1}-M_{2i})\right]-1, M_{2i+1}\right\}
\]
and
\[
S_{2i-1}=\left\{M_{2i-1}+1, \cdots, M_{2i-1}+\left[\frac{2i-1}{2i} M_{2i}\right]-L_{2i-1}-1, M_{2i}\right\},
\]
where
\[
L_{2i-1}=\sum_{j=1}^{i-1}\left[\frac{2j-1}{2j}M_{2j}\right]+\sum_{j=1}^{i-1}\left[\frac{2j}{2j+1}(M_{2j+1}-M_{2j})\right].
\]
Define
\[
S=\bigcup_{i\ge1}S_i. \qedhere
\]

\end{proof}


\subsection*{Acknowledgements}
The authors would like to thank Xianghong Chen for valuable discussions at the early stage of this work, especially for drawing our attention to Proposition \ref{pro}. This project was supported by the National Natural Science Foundation of China (11671189, 11771153 \& 11471124), the Natural Science Foundation of Fujian Province (2017J01403), and the
Program for New Century Excellent Talents in Fujian Province University.


\begin{thebibliography}{1}

\bibitem{Ass}P. Assouad,\textrm{ Plongements lischitziens dans $\mathbb{R}^{n}$,} Bull. Soc. Math. France. \textbf{111} (1983), 429--448.
\bibitem{BPe} R. Balka and Y. Peres, \textrm{ Uniform dimension results for fractional Brownian motion,}  J. Fractal Geom. \textbf{4} (2017), 147--183.
\bibitem{BP} C. Bishop and Y. Peres, Fractal Sets in Probability and Analysis, Cambridge Studies in Advanced Mathematics, vol. 162. Cambridge: Cambridge University Press, 2017.
\bibitem{Car}M. Carnovale, \textrm{Long progressions in sets of fractional dimension,} Preprint.
\bibitem{CLP}V. Chan, I. {\L}aba and M. Pramanik, \textrm{Finite configurations in sparse sets,} J. Anal. Math. \textbf{128} (2016), 289--335.
\bibitem{DWW} Y. X. Dai, C. Wei and S. Y. Wen, \textrm{Some geometries properties of sets defined by digit restrictions,} Int. J. Number Theory \textbf{13} (2017), 65--75.
\bibitem{DZ} S. Dyaylov and J. Zahl, \textrm{Spectral gaps, additive ennergy, and a fractal uncertainty principle,} Geom. Funct. Anal. \textbf{26} (2016), 1011--1094.
\bibitem{ET} P. Erd\H{o}s and P. Tur\'{a}n, \textrm{On Some Sequences of Integers,} J. London Math. Soc. \textbf{11} (1936), 261--264.
\bibitem{Fal} K. J. Falconer, Fractal Geometry-Mathematical Foundations and Applications. John Wiley and Sons Ltd, Chichester, 1990.
\bibitem{FY} J. M. Fraser and H. Yu, \textrm{Arithmetic patches, weak tangents, and dimension,} Bull. London Math. Soc. DOI: 10.1112/blms.12112.
\bibitem{FSY} J. M. Fraser, K. Saito and H. Yu,  \textrm{Dimensions of sets which uniformly avoid arithmetic progressions,} Preprint.
\bibitem{Fur}  H. Furstenberg. Recurrence in Ergodic Theory and Combinatorial Number Theory. Princeton Univ. Press, 1981.
\bibitem{GT}B. Green and T. Tao, \textrm{The primes contain arbitrarily long arithmetic progressions,} Ann. Math. \textbf{167} (2008), 481--547.
\bibitem{Hei}J. Heinonen, Lectures on Analysis on Metric Spaces, Springer-Verlag, New York, 2001.
\bibitem{LP}I. {\L}aba and M. Pramanik, \textrm{Arithmetic progressions in sets of fractional dimension,} Geom. Funct. Anal. \textbf{19} (2009), 429--456.
\bibitem{CKL}C.-K. Lai,  \textrm{Perfect fractal sets with zero Fourier dimension and arbitrarily long arithmetic progressions,} Ann. Acad. Sci. Fenn. Math. \textbf{42} (2017), 1009--1017.
\bibitem{LLMX} W. W. Li, W. X. Li, J. J. Miao, L. F. Xi. \textrm{Assouad dimensions of Moran sets and Cantor-like sets},  Front. Math. China. \textbf{11} (2016), 705--722.
\bibitem{LW2} J. J. Li and M. Wu, \textrm{On exceptional sets in  Erd\H{o}s-R\'{e}nyi limit theorem revisited,} Monatsh Math. \textbf{182} (2017), 865--875
\bibitem{Luu}J. Luukkainen, \textrm{Assouad dimension: Antifractal metrization, porous sets, and homogeneous measures,} J. Korean Math. Soc. \textbf{35} (1998), 23--76.
\bibitem{Ma}P. Mattila. Fourier Analysis and Hausdorff Dimension. Cambridge University Press, Cambridge, 2015.
\bibitem{Po}P. Potgieter, \textrm{Arithmetic progressions in Salem-type subsets of the integers,} J. Fourier Anal. Appl. \textbf{17} (2011), 1138--1151.
\bibitem{Po2}P. Potgieter, \textrm{Salem sets, equidistribution and arithmetic progressions,} Preprint.
\bibitem{Rob}J. C. Robinson, Dimensions, Embeddings, and Attractors, Cambridge University Press, 2011.
\bibitem{Ro} K. Roth, \textrm{On certain sets of integers,} J. London Math. Soc. \textbf{28 }(1953), 245--252.
\bibitem{Shm} P. Shmerkin,  \textrm{Salem sets with no arithmetic progression,} Internat. Math. Res. Notices.  \textbf{7} (2017), 1929--1941.
\bibitem{SZ} R. Salem and A. Zygmund, \textrm{Sur un th\'{e}or\`{e}me de Piatet\c{c}ki-Shapiro,}  C. R. Acad. Sci. Paris, \textbf{240} (1955), 2040--2042.
\bibitem{Sz1} E. Szemer\'{e}di, \textrm{On sets of integers containing no four elements in arithmetic progression,}  Acta Math. Acad. Sci. Hungar. \textbf{20} (1969), 89--104.
\bibitem{Sz2} E. Szemer\'{e}di, \textrm{On sets of integers containing no $k$ elements in arithmetic progression,} Acta Arith. \textbf{27} (1975), 299--345.
\end{thebibliography}
\end{document}